\documentclass[11pt]{amsart}
\usepackage[utf8]{inputenc}
\usepackage[T1]{fontenc}
\usepackage{fixltx2e}
\usepackage{graphicx}
\usepackage{longtable}
\usepackage{float}
\usepackage{wrapfig}
\usepackage[normalem]{ulem}
\usepackage{textcomp}
\usepackage{marvosym}
\usepackage[nointegrals]{wasysym}
\usepackage{latexsym}
\usepackage{amssymb}
\usepackage{amstext}
\usepackage{hyperref}
\usepackage[all]{pabmacros}
\usepackage{amsmath}
\usepackage[bibstyle=alphabetic,citestyle=alphabetic,backend=bibtex]{biblatex}
\bibliography{refs}
\AtEveryBibitem{\clearfield{doi}}
\AtEveryBibitem{\clearfield{url}}
\AtEveryBibitem{\clearfield{issn}}
\renewcommand*{\bibfont}{\footnotesize}
\DeclareMathOperator{\reflectionvector}{\vec{V}}
\newcommand{\reflectionplane}[1][\reflectionvector]{\ensuremath{P_{#1}}}
\newcommand{\reflectionmap}[1][\reflectionvector]{\ensuremath{R_{#1}}}
\newcommand{\reflectionset}[2][\reflectionvector]{\ensuremath{{#2}_{#1}}}
\newcommand{\reflectionhalfspace}[1][\reflectionvector]{\ensuremath{\reflectionset[{#1}]{H}}}
\usepackage{snapshot}
\address{Department of Mathematics, University of California San Diego}
\email{pbryan@ucsd.edu}
\email{jalouie@math.ucsd.edu}
\keywords{Curve shortening flow, Ancient solutions, Aleksandrov reflection, Harnack}
\subjclass[2010]{53C44, 35K55, 58J35}
\date{}
\author{Paul Bryan and Janelle Louie}
\title[Convex Ancient CSF on $\sphere^2$]{Classification of Convex Ancient Solutions to Curve Shortening Flow on the Sphere}
\hypersetup{
  pdfkeywords={},
  pdfsubject={},
  pdfcreator={Emacs 24.3.1 (Org mode 8.2.4)}}
\begin{document}

\maketitle
\begin{abstract}
We prove that the only closed, embedded ancient solutions to the curve shortening flow on $\sphere^2$ are equators or shrinking circles, starting at an equator at time $t=-\infty$ and collapsing to the north pole at time $t=0$. To obtain the result, we first prove a Harnack inequality for the curve shortening flow on the sphere. Then an application of the Gauss-Bonnet, easily allows us to obtain curvature bounds for ancient solutions leading to backwards smooth convergence to an equator. To complete the proof, we use an Aleksandrov reflection argument to show that maximal symmetry is preserved under the flow.
\end{abstract}

\section{Introduction}
\label{sec-1}

In this paper we study a time-dependent family of smooth, embedded, closed curves $\gamma_t = F_t(\sphere^1) \subset \sphere^2$ on the unit sphere, evolving by the curve shortening flow:
\begin{equation}
\label{eq:csf}
\pd[F_t]{t} = -\curvecurv \nor
\end{equation}
where $\curvecurv$ is the (signed) geodesic curvature of $\gamma_t$ with respect to a choice of smooth unit normal vector field $\nor$. Our aim is to classify convex ($\curvecurv > 0$) ancient solutions, which by definition exist on the maximal time interval $(-\infty,T)$. If $T\ne \infty$, we will assume from now on that $T=0$. We prove the following theorem:

\begin{theorem}
[Classification of Ancient Solutions]
\label{thm:classification_ancient}
Let the family of curves $\gamma_t$ be a closed, convex, embedded ancient solution to the curve shortening flow on the sphere $S^2$. Then $\gamma_t$ is either a fixed equator for all $t \in (-\infty, \infty)$, or a family of shrinking geodesic circles, existing on $(-\infty, 0)$, converging to an equator as $t \to -\infty$ and shrinking to a point at $T=0$.
\end{theorem}

The curve shortening flow has been studied extensively in the plane. The principal result is the Gage-Hamilton-Grayson theorem \cite{MR840401, MR906392}, stating that arbitrary, smooth, embedded, closed solutions $\gamma_t$ shrink to round points in finite time $T<\infty$. On surfaces, the curves  shortening flow either collapses to a round point in finite time (as in the plane case), or exists for all time, converging to a closed geodesic as $t\to\infty$ \cite{MR979601,MR1630194,MR2668967}.

Ancient solutions to the curve shortening flow in the plane have been classified in \cite{MR2669361} as precisely the contracting circles and the contracting Angenent ovals. The former is a Type I ancient solution ($\lim_{t\to-\infty} \sup_{\gamma_t}\abs{\curvecurv} < \infty$), while the latter is a Type II ancient solution ($\lim_{t\to-\infty} \sup_{\gamma_t}\abs{\curvecurv} = \infty$). Our theorem shows that the only Type I ancient solution on the sphere are the "obvious" ones, and that no Type II ancient solutions exist.

To begin in section \ref{sec-2}, we introduce some notation and preliminary results. Then in \ref{sec-3}, we obtain a Harnack inequality for convex curves evolving by curve shortening on $\sphere^2$ in Theorem \ref{thm:harnack}. This allows us to show that for convex curves, the curvature is monotonically increasing, hence bounded on any time interval $(-\infty, t_0]$. A standard bootstrapping argument then furnishes us with bounds on all higher derivatives. Next, in section \ref{sec-4} we use the Gauss-Bonnet theorem to show that $\int_{\gamma_t}\curvecurv \to 0$ as $t\to -\infty$. Combining this with the curvature estimates, we are readily able to show that $\gamma_t$ converges smoothly to an equator as $t\to-\infty$. To complete the theorem, in section \ref{sec-5}, we use a perturbed, parabolic version of Aleksandrov reflection inspired by \cite{MR1846204,MR1386736}. This shows that $\gamma_t$ reflects "above" (the precise definition is given in section \ref{sec-5}) itself for the perturbed reflection and hence so too in the limit for all reflections preserving the equator. It is then easy to show that each $\gamma_t$ is preserved under all equator preserving reflections and therefore is a round circle.

\section*{Acknowledgements}
Both authors would like to thank Professor Bennett Chow for suggesting this problem and providing much useful guidance on laying out the program. The second author is especially thankful, this paper arising from her Ph.D. thesis under Professor Chow's supervision. This paper was completed while the first author was a SEW Visiting Assistant Professor at UCSD, acting as an informal Ph.D. advisor to the second author's Ph.D. research at UCSD.
\section{Notation and Preliminaries}
\label{sec-2}
\label{sec:notation}

\subsection{Convex Curves on $\sphere^2$.}
\label{sec-2-1}

A closed, embedded curve $\gamma$ divides $\sphere^2$ into two open, disjoint regions. If one region has area strictly small than $2\pi$, we label that region $\interior{\Omega}$ and call it the interior of $\gamma$. The other region $\exterior{\Omega}$ is the exterior. Let $\nor$ denote the interior pointing unit normal (so that for $x\in\gamma$, $\exp_x^{\sphere^2}(\epsilon\nor) \in \interior{\Omega}$ for small $\epsilon > 0$). Note that if the area of both regions equals $2\pi$, it is not possible in general to single out one region as interior and one as exterior; consider for instance when $\gamma$ is an equator. The issue is equivalent to defining a unit normal vector field on $\gamma$ and declaring it be either interior or exterior pointing. In such a case, we will choose a unit normal vector field $\nor$ and designate it interior pointing.

On a Riemannian manifold $M$, there are several notions of convexity in common use. We will use the following definitions: A subset $K \subset M$ is \emph{(geodesically) convex} if every two points $x,y \in K$ can be joined by a length minimising geodesic (of $M$) entirely contained within $K$. Note that we don't require this geodesic to be unique so that a closed hemisphere of $\sphere^n$ is geodesically convex. $K$ is \emph{weakly (geodesically) convex} if any two points in $K$ may be connected by a length minimizing geodesic of $K$. The difference between the two notions is that the length minimising geodesic in a weakly convex set need not be length minimising in $M$. It is well known that weakly convex is equivalent to non-negative boundary curvature. For example, if $M = \sphere^1 \times \RR$ is a flat cylinder, then a geodesic disc of radius greater than $\pi/2$ is weakly convex, but not convex. On the sphere however (as in the plane), weakly convex sets are convex (Proposition \ref{prop:convex_sets}). Since geodesics in $\sphere^2$ are great circles, and length minimising geodesics are half great circles, the length of any minimizing geodesic joining $x$ to $y$ is at most $\pi$. The following proposition characterises convex regions $K$ of $\sphere^2$ with boundary a smooth embedded curve. The results (and arguments) are standard and well known, but we could not find a good single reference, so give the details here.

\begin{prop}
\label{prop:convex_sets}
Let $K \subset \sphere^2$ be a connected open set with boundary $\bdry{K} = \gamma$ a smooth, closed  embedded curve. Let $\nor$ be the interior unit normal vector field along $\gamma$ and $\curvecurv$ the geodesic curvature with respect to $\nor$. Then the following are equivalent:

\begin{enumerate}
\item $K$ is convex,
\item for every $x \in \gamma$, $K \subset H^+_x$ where $H^+_x$ is the hemisphere with boundary the tangent great circle $E_x$ to $\gamma(x)$ and interior normal $\nor(x)$,
\item $\gamma$ may be written as the graph over an equator of a smooth function $f$ such that 
   \[
   f''(\theta) + 2 \tan(f(\theta)) (f'(\theta))^2 + \cos (f(\theta)) \sin (f(\theta)) \geq 0, \quad \theta \in \sphere^1.
   \]
\item The boundary curvature, $\curvecurv \geq 0$.
\end{enumerate}
\end{prop}

\begin{proof}
\begin{itemize}
\item (1) $\Rightarrow$ (2):

Let $x \in \gamma$ and let $E_x$ be the tangent great circle to $\gamma = \bdry{K}$ at $x$. With $\nor(x)$ pointing interior to $K$, let $H^+(x)$ denote the hemisphere with boundary $E_x$ and interior normal $\nor(x)$. Then $H^+(x) \intersect K \ne \emptyset$ and we need to show that in fact $K \subset H^+(x)$. 

Suppose otherwise, so that there is a $y$ in the interior of $H^-$, and let $\sigma$ be the unique (since $y \notin E_x$) length minimising geodesic joining $x$ to $y$. $\sigma$ lies in $H^-$, but also does not intersect $K$ in a neighbourhood of $x$ since it points exterior to $K$ ($\ip{\sigma'(x)}{\nor(x)} < 0$). On the other hand, since $K$ is convex and both $x,y\in K$ we must have $\sigma \subset K$, a contradiction.

\item (2) $\Rightarrow$ (3):

We work in polar coordinates $(\cos\phi\cos\theta, \cos\phi\sin\theta, \sin\phi)$ with the equator given by $\phi = 0$ (as opposed to the usual convention of $\phi=0$ at the north pole). For graphs $(\theta, f(\theta))$ the curvature is given by
\[
  \curvecurv = \frac{\cos(f)}{((f')^2 + \cos^2(f))^{3/2}} \left(f'' + 2(f')^2 \tan(f) + \sin (f) \cos(f)\right).
  \]
So here, we prove that (2) implies $\gamma$ may be written as a graph and that $\curvecurv \geq 0$.

\emph{$\gamma$ is a graph}: Take any tangent great circle $E$, $H^+$ the hemisphere with $K \subset H^+$, and $p$ the center of $H^+$. If $p \in \gamma$, then we can rotate $E$ along the geodesic joining $p$ to $x$ to obtain a new $E$ with $K \subset H^+$ and $p \in K$ (remember $K$ is open). Then by convexity, for any $y \in E$, the geodesic ray joining $p$ to $y$ intersects $\gamma$ in precisely one point, which we may denote by $f(y)$, expressing $\gamma$ as the graph of $f$ (which must be smooth).

$\curvecurv \geq 0$: This follows since $K$ lies on one side of every tangent great circle, so the local Taylor expansion of $\gamma$ shows that the curvature vector is interior pointing everywhere.

\item (3) $\Rightarrow$ (4):

As above, the condition on $f$ is precisely that $\curvecurv \geq 0$.

\item (4) $\Rightarrow$ (1):

We prove the contrapositive. Suppose that $K$ is not convex. We need to show that there is a $y \in \gamma$ such that $\curvecurv(y) < 0$.

Since $K$ is not convex, there is a length minimising geodesic $\alpha$ meeting $\clsr{K}$ only at it's endpoints. Moreover, we can choose $\alpha$ to have length strictly less than $\pi$: if not, then any arbitrary great circle must intersect $\cmplt{K}$ in a connected arc of length at least $\pi$ and so intersects $K$ in a connected arc of length at most $\pi$. Therefore any $x,y \in K$ may be connected by a length minimising geodesic contradicting that $K$ is not convex.

Let $\sigma_z$ be the continuous family of geodesic rays of length $\pi$ starting at $z \in \alpha$, perpendicular to $\alpha$. There are precisely two such families, and we choose $\sigma_z$ so that $\sigma_z$ intersects $\gamma$ at distance less than $\pi$ for $z$ near the endpoints of $\alpha$. Note that on the sphere, for each $z \in \alpha$, we have $\sigma_z \intersect \gamma \ne \emptyset$ since the end points of $\alpha$ (which also lie on $\gamma)$ lie on either side of the equator containing $\sigma_z$. Let $y = y(z) \in \gamma$ be the first point where $\gamma_z$ meets $\gamma$ and let $\rho(z) = d(z,y)$. 

Then $\rho$ is continuous and attains it's maximum at a point $z_0$ in the interior of $\alpha$ since $\rho = 0$ on the end points of $\alpha$ and by assumption $\rho(z) > 0$ for some $z \in \alpha$. Now we have $\rho_{z_0}$ a geodesic meeting $\gamma$ orthogonally at $y_0$, hence we can solve for $z$ as a function of $y$ near $y_0$. Then the second variation formula (varying $y$)) shows that $\curvecurv(y_0) \leq 0$.
\end{itemize}
\end{proof}

From here on we will freely use the results of the proposition without further comment and by a \emph{convex curve} we will mean a closed, embedded curve $\gamma$ satisfying any of the four conditions. Lastly, let us note that for $\gamma$ convex, approximating $\gamma$ by convex polygons (with geodesic arcs), it is possible to show that the length $L(\gamma) \leq 2\pi$, a fact we will employ in section 4. See \cite[Problem 1.10.1]{MR2208981} for details. In section 5, we will find it very useful to write $\gamma$ as the graph over an equator.
\subsection{Evolution of basic quantities}
\label{sec-2-2}

Let us now record the evolution of various quantities under the curve shortening flow. This is very similar to the plane case \cite{MR840401,MR742856}. We make use of the Serret-Frenet formulae,
\[
\conx_{\tang} \tang = \curvecurv \nor, \quad \conx_{\tang} \nor = -\curvecurv \tang
\]
with $\tang$ the unit tangent to $\gamma$ and $\nor$ the interior pointing normal.

Let $s = s_t$ denote the arc-length parameter of $\gamma_t$. The commutator of $\pd{s}$ and $\pd{t}$ is
\begin{equation}
\label{eq:commutator}
\left[\pd{t}, \pd{s}\right] = -\curvecurv^2 \pd{s}.
\end{equation}

Under the curve shortening flow on $\sphere^2$, the curvature evolves according to
\begin{equation}
\label{eq:curvature_evolution}
\curvecurv_t = \curvecurv_{ss} + \curvecurv^3 + \curvecurv
\end{equation}
where subscripts denote partial derivatives. The maximum principle now ensures that if $\curvecurv > 0$ at some time $t_0$, then this holds for all $t\geq t_0$. 

Lastly, the element of arc-length $ds$ evolves according to
\begin{equation}
\label{eq:arclength_evolution}
\pd{t} ds = -\curvecurv^2 ds.
\end{equation}
\subsection{Aleksandrov reflection}
\label{sec-2-3}

In section 5 we will make use of an Aleksandrov reflection argument, so give the preliminaries here. Embed $\sphere^2$ in $\RR^3$ via the standard embedding, and let $E$ denote the equator $\{z = 0\}$. 

The argument rests on a "tilted" Aleksandrov reflection. Let $\reflectionvector$ be a vector in $\RR^3$ such that $\ip{\reflectionvector}{\vec{e}_z} < 0$ where $\vec{e}_z = (0,0,1)$, and let $\reflectionplane$ be the plane through the origin with normal vector $\reflectionvector$. Then $\reflectionplane$ intersects $E$ in two antipodal points. Let $\delta(V) \in (0, \pi/2)$ be the angle between $\reflectionvector$ and the plane $\{z=0\}$. Notice that for each fixed $\delta$, the set of $\reflectionvector$ with $\delta(\reflectionvector)=\delta$ is a compact set parameterised by $\sphere^1$ acting as rotations about the $z$-axis. We consider the Aleksandrov reflection across the plane $P$,
\[
\reflectionmap(X) = X - 2\ip{X}{V}V
\]
which is an isometry of $\RR^3$ preserving $\sphere^2$, hence is also an isometry of $\sphere^2$. See figure \ref{fig:tilted_reflection}.

\begin{figure}[htb]
\centering
\includegraphics[width=.9\linewidth]{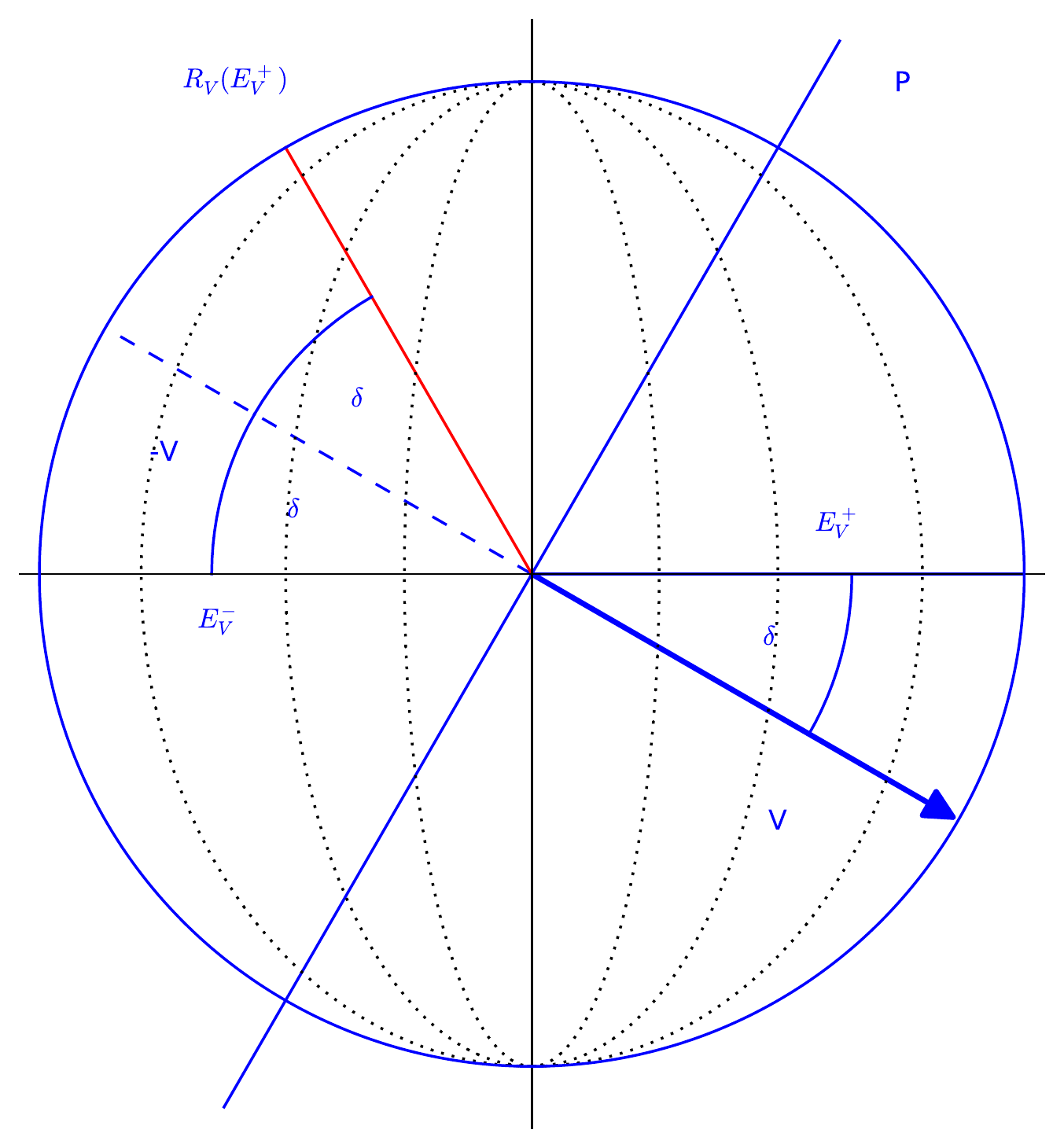}
\caption{\label{fig:tilted_reflection}Reflection in the $(\vec{e_z}, V)$-plane showing the reflected equator and geodesics through the north pole (dotted lines).}
\end{figure}

Let $\reflectionhalfspace^+ = \{X \in \RR^3: \ip{X}{\reflectionvector}>0\}$ be the open half space lying on the side of $\reflectionplane$ into which $\reflectionvector$ points, and $\reflectionhalfspace^- = \{X \in \RR^3: \ip{X}{\reflectionvector} < 0\}$ be the open half space on the other side of $\reflectionplane$. For any set $S\subset \RR^3$, let $\reflectionset{S}^{\pm} = S \intersect \reflectionhalfspace^{\pm}$. In particular $\reflectionset{(\sphere^2)}^{\pm}$ are hemispheres with boundary equator equal to $\reflectionplane \intersect \sphere^2$. Notice that $\reflectionmap$ takes $\reflectionhalfspace^{\pm}$ to $\reflectionhalfspace^{\mp}$.

Next we will need the nearest point projection $\pi(X)$ to the equator $E$. For a point $X\in\sphere^2$, let $\rho(X) = d(X, E)$ denote the distance from $X$ to $E$ and let $\pi(X)$ denote the set of nearest points on $E$ to $X$. If $X$ is not either pole $(0,0,1), (0,0,-1)$, then $\pi(X)$ is a single point. Otherwise, $\pi(X) = E$. For each point $y\in \pi(X)$, there is a unique, length minimising geodesic joining $X$ to $y$ with length equal to $\rho(X)$. By the first variation formula, this geodesic meets $E$ orthogonally, hence lies on the unique great circle passing through $X$ and the north pole $(0,0,1)$. The dotted lines in figure \ref{fig:tilted_reflection} show some geodesics passing through the north pole.

For any two curves $\alpha,\beta$ on $\sphere^2$, and any $X \in E$, let us write $\alpha \geq_X \beta$ (resp. $\alpha >_X \beta$) if
\[
\inf \{\rho(Y) : Y \in \pi^{-1}(X) \intersect \alpha\} \geq (\text{resp. } >) \> \sup \{\rho(Y) : Y \in \pi^{-1}(X) \intersect \beta\}
\]
whenever both sets are non-empty. The $\inf$ and $\sup$ are required since $\alpha$ and $\beta$ need not be graphs over the equator and so the fibres $\pi^{-1}(X) \intersect \alpha$ and $\pi^{-1}(X) \intersect \beta$ may have multiple points. Loosely speaking, we say $\alpha$ lies above $\beta$ over the point $X$ in the equator $\{z=0\}$. We will also write $\alpha \geq (\text{resp. } >) \> \beta$ if $\alpha \geq_X (\text{resp. } >_X) \> \beta$ for every $X \in E$. Notice in particular that we require strict inequality to hold for \emph{every} $X$.

\begin{remark}
\label{rem:partial_order}
The relations $\leq_X$ and $\leq$ are not partial orders in general since they are not reflexive. In fact, $\alpha \leq_X \alpha$ if and only if the fibre $\pi^{-1} (X)$ intersects $\alpha$ in a single point. The relation $\leq$ is only a partial order when restricted to curves that are graphs over the equator: $\alpha \leq \alpha$ if and only if $\alpha$ is a graph over the equator. 
\end{remark}

Using polar coordinates as above, we may also rewrite $\alpha \geq_X \beta$ if and only if $\theta(X) \in \{\theta(\alpha)\} \intersect \{\theta(\beta)\}$ and
\[
\inf\{\phi(\alpha) : \theta(\alpha) = \theta(X)\} \geq \sup\{\phi(\beta) : \theta(\beta) = \theta(X)\}.
\]
That is, $\alpha \geq_X \beta$ if and only if there is at least one point on $\alpha$ and at least one point on $\beta$ with azimuthal angle $\theta$ equal to the azimuthal angle of $X$ and so that the smallest polar angle $\phi$ of $\alpha$ is greater than or equal to the greatest polar angle of $\beta$.
\section{Harnack Inequality and Curvature Estimates}
\label{sec-3}

Just as for the curve shortening flow in the plane, there is a Harnack inequality for the curve shortening flow on $\sphere^2$. This is the fundamental result of this section, from which everything else follows.

\begin{theorem}
[Harnack Inequality]
\label{thm:harnack}
For any immersed solution to the curve shortening flow defined on the time interval $[-\alpha, 0)$ and  with $\curvecurv > 0$, we have
\[
(\log k)_{ss} + k^2 + \frac{1}{2(t-\alpha)} \geq 0.
\]
\end{theorem}

\begin{proof}
From the evolution of the curvature in equation \eqref{eq:curvature_evolution} and the commutator equation \eqref{eq:commutator}, we deduce
\begin{align*}
k_{st} &= k_{ts} + k^2k_s = (k_{sss} + 3k^2k_s + k_s) + k^2k_s\\
&= k_{sss} + 4k^2k_s + k_s \\
k_{sst} &= k_{tss} + 2kk_s^2 + 2k^2k_{ss} = (k_{ssss} + 3k^2k_{ss} + 6kk_s^2 + k_{ss}) + 2kk_s^2 + 2k^2k_{ss} \\
&= k_{ssss} + 5k^2k_{ss} + 8kk_s^2 + k_{ss}.
\end{align*}

Let $Q$ be the quantity
\[
Q = (\log k)_{ss} + k^2.
\]

Computing the time derivative, we get
\begin{align*}
Q_t &= -\frac{k_{ss}}{k^2}k_t + \frac{k_{sst}}{k} + 2k^{-3}k_t k_s^2 - k^{-2}(2k_sk_{st}) + 2kk_t \\
&= -\frac{k_{ss}^2}{k^2} + 6kk_{ss} + \frac{k_{ssss}}{k} + 2k_s^2 + \frac{2k_s^2k_{ss}}{k^3} - \frac{2k_sk_{sss}}{k^2} + 2k^4 + 2k^2 \\
&= Q_{ss} + \left(\frac{2k_s}{k}\right)Q_s +2Q^2 +2k^2 \\
&\ge Q_{ss} + \left(\frac{2k_s}{k}\right)Q_s + 2Q^2.
\end{align*}

An ODE comparison with $q(t) = -1/2(t-\alpha)$ which satisfies $q_t = 2q^2$ and $\lim_{t\to\alpha} q(t) = -\infty$ shows that $Q(s,t) \geq q(t)$ completing the result.
\end{proof}

\begin{cor}
\label{cor:curvature_time_increasing}
For an ancient solution $\gamma_t$, with $\curvecurv > 0$, we have
\[
\curvecurv_t \geq 0.
\]
\end{cor}

\begin{proof}
As $\gamma_t$ is an ancient solution with $\curvecurv>0$, the Harnack inequality holds for any $\alpha < 0$, supplying us with
\begin{align*}
0 &\leq (\log k)_{ss} + k^2 + \frac{1}{2(1-\alpha)} \\
&= \frac{\curvecurv_{ss}}{\curvecurv} - \frac{\curvecurv_s^2}{\curvecurv} + \curvecurv^2 + \frac{1}{2(1-\alpha)} \\
& \leq \frac{\curvecurv_{ss} + \curvecurv^3 + \curvecurv}{\curvecurv} + \frac{1}{2(1-\alpha)} \\
&= \frac{\curvecurv_t}{\curvecurv} + \frac{1}{2(1-\alpha)} \\
\end{align*}
for $t \in [\alpha, 0)$. Taking the limit $\alpha \to -\infty$ gives the result.
\end{proof}

We are now able to obtain a curvature bound, and by standard bootstrapping arguments, we also obtain higher derivative bounds.

\begin{cor}
\label{cor:derivative_bounds}
For any $t_0 < 0$ and any integer $j\geq 0$, there exists a constant $C_j(t_0)$ such that
\[
\abs{\curvecurv^{(j)}} \leq C_j(t_0)
\]
on $(-\infty, t_0)$.
\end{cor}

\begin{proof}
By Corollary \ref{cor:curvature_time_increasing}, $\curvecurv$ is increasing in $t$, and since $\curvecurv > 0$, we may take $C_0(t_0) = \sup \{\curvecurv(x,t_0) : x\in \sphere^1\}$. 

The higher derivative estimates follow by standard bootstrapping arguments similar to those described in \cite{MR1375255}. For example, we obtain $C_1(t_0)$ from the evolution equation $(k_s)_t = k_{sss} + 4k^2k_s + k_s$ by applying the maximum principle to the evolution of $(t-(t_0-1))k_s^2 + C_0(t_0) k^2$ and using the fact that $\abs{\curvecurv} \leq C_0(t_0)$.
\end{proof}
\section{Backwards Convergence}
\label{sec-4}

Armed with the curvature bounds, we can prove that any ancient solution $\gamma_t$ converges smoothly to an equator as $t\to-\infty$. We begin with a lemma.

\begin{lemma}
\label{lem:intcurvetozero}
Let $\gamma_t$ be an ancient solution to the curve shortening flow. Then
\begin{align*}
\lim_{t\to -\infty} \int_{\gamma_t} k ds = 0
\end{align*}
exponentially fast.
\end{lemma}

\begin{proof}
By the Gauss-Bonnet theorem we have
\[
\int_{\gamma_t} \curvecurv ds = 2\pi - A
\]
where $A$ is the area of $\interior{\Omega_t}$. Recalling that $\curvecurv_t = \curvecurv_{ss} + \curvecurv^3 + \curvecurv$ and that $(ds)_t = - \curvecurv^2 ds$, we obtain
\[
\pd{t} \int_{\gamma_t} \curvecurv ds = \int_{\gamma_t} \curvecurv_{ss} + \curvecurv ds = \int_{\gamma_t} \curvecurv ds = 2\pi - A.
\]
Therefore
\[
A_t = A - 2\pi,
\]
and hence
\begin{equation}
\label{eq:At}
A = 2\pi[1 - (1 - A(0)/2\pi)e^t].
\end{equation}
Then letting $t\to-\infty$, we find that $A(t) \to 2\pi$ exponentially fast and the Gauss-Bonnet formula implies that
\[
\int_{\gamma_t} \curvecurv ds \to 0
\]
exponentially fast.
\end{proof}

Combining our estimates so far, we now obtain the following important proposition:

\begin{prop}
\label{prop:dk_curv_tozero}
Let $\gamma_t$ be an ancient solution to the curve shortening flow. Then for every integer $j\geq 0$, we have
\[
\max_{s \in \sphere^1} \abs{\curvecurv^{(j)}} (s, t) \to 0
\]
as $t \to -\infty$.
\end{prop}

\begin{proof}
First, let us prove the case $j=0$. We argue by contradiction. Suppose the proposition is false. Then there exists $\epsilon>0$, a sequence $t_i \to -\infty$ and, a sequence $s_i$ such that $\curvecurv(s_i,t_i) \ge \epsilon$ for all $i$. Since we also have $\abs{\curvecurv_s} \le C_1(1)$ for $t\leq 1$, we find that 
\[
\curvecurv(s, t_i) \geq \epsilon -C_1(1)\abs{s-s_i} \geq \epsilon/2\]
for all $s$ such that $|s-s_i| \leq \frac{\epsilon}{2C_1(1)}$. But this implies that for every $i$,
\[
\int_{\gamma_{t_i}} \curvecurv (s,t) ds \geq \int_{|s-s_i| \leq \tfrac{\epsilon}{2C_1(1)}} \curvecurv(s,t)ds \geq \frac{\epsilon^2}{4C_1(1)}
\]
contradicting the fact that $\int_{\gamma_t} \curvecurv \to 0$ as $t\to -\infty$ by lemma \ref{lem:intcurvetozero}.

The result for $j>0$ follows by a bootstrapping argument similar to the proof of Corollary \ref{cor:derivative_bounds}.
\end{proof}

Now we have all the ingredients to prove that the backwards limit is an equator. First we have sub-sequential convergence.

\begin{lemma}
\label{lem:subsequential_backward_limit}
Let $\gamma_t$ be an ancient, embedded, convex solution to the curve shortening flow on $\sphere^2$. Then there is a sequence $t_k \to -\infty$ with $\gamma_{t_k} \to_{C^{\infty}} \gamma_{-\infty}$ as $k\to\infty$, with $\gamma_{-\infty}$ an equator.
\end{lemma}

\begin{proof}
Since $\pd{t} L = -\int \curvecurv^2 ds < 0$, $L(t)$ is bounded below by $L(-1) > 0$ for all $t\leq-1$. By the paragraph following Proposition \ref{prop:convex_sets}, $L \leq 2\pi$.

Proposition \ref{prop:dk_curv_tozero} shows that the curvature and all derivatives converge to $0$. Since $L$ is bounded, as $t\to -\infty$, $\abs{\gamma'}$ is bounded (say in a constant speed parametrisation) above and away from zero. The Arzela-Ascoli theorem then provides us with a sequence $t_k \to -\infty$ such that $\gamma_{t_k}$ converges smoothly to a closed, immersed curve with zero curvature, i.e. to an equator $\gamma_{-\infty}$.
\end{proof}

Next we show that the limit is unique and that the flow remains in a fixed hemisphere.

\begin{cor}
\label{cor:hemisphere}
The equator $\gamma_{-\infty}$ is unique and $\gamma_t$ lies in one of the hemispheres $H^{\pm}_{-\infty}$ defined by $\gamma_{-\infty}$ for all $t \in (-\infty,0)$. 
\end{cor}

\begin{proof}
Since $\curvecurv > 0$, Gauss-Bonnet implies that $A(t) < 2\pi$ for all $t$ and that the curvature vector points inward (recall that the interior $\interior{\Omega_t}$ is the region enclosed by $\gamma_t$ with area less than $2\pi$). Thus $\interior{\Omega_{t_1}} \subsetneq \interior{\Omega_{t_2}}$ whenever $t_2 < t_1$. 

In particular, for our sequence $(t_k)$, $\interior{\Omega_{t_j}} \subsetneq \interior{\Omega_{t_k}}$ for $k>j$. If $\interior{\Omega_{t_j}}$ is not wholly contained in either hemisphere $H_{-\infty}^{\pm}$ for some $j$, then it contains points in both hemispheres, $x_j^{\pm} \in H_{-\infty}^{\pm}$. We obtain a contradiction by choosing $k<j$ with $\gamma_{t_k}$ sufficiently close to $\gamma_{-\infty}$ so that $x_j^+$ lies on the opposite side of $\gamma_{-\infty}$ to $x_j^-$ contradicting $\interior{\Omega_{t_k}} \subset \interior{\Omega_{t_j}}$ has points on both sides of $\gamma_{t_k}$. Thus $\interior{\Omega_{t_k}}$ lies entirely in one or the other hemisphere $H_{-\infty}^{\pm}$ for every $k$.

Now for any $t \in (-\infty,0)$, choose $k$ such that $t_k < t$. Then $\interior{\Omega_{t}} \subsetneq \interior{\Omega_{t_k}}$, the latter lying in a hemisphere. Lastly, suppose there is a sequence $t_k'$ with $\gamma_{t_k'}$ converging to different equator. This equator has points lying in both hemispheres defined by $\gamma_{-\infty}$ and hence $\gamma_{t_k'}$ also has points in both hemispheres for $t_k'$ sufficiently negative, a contradiction.
\end{proof}

\begin{remark}
Any closed, embedded curve on $\sphere^2$ with $\curvecurv \geq 0$ must lie in a closed hemisphere. The result above shows that under the flow, an embedded, convex, ancient solution remains in a fixed hemisphere for all time.
\end{remark}

Now we can extend the sub-sequential convergence to full convergence.

\begin{theorem}
\label{thm:backward_limit}
Let $\gamma_t$ be an ancient, embedded, convex solution to the curve shortening flow on $\sphere^2$. Then $\gamma_{t} \to_{C^{\infty}} \gamma_{-\infty}$ up to diffeomorphism as $t \to -\infty$.
\end{theorem}

\begin{proof}
First, we have $C^0$ convergence: for any $\epsilon>0$, choose $t_k$ with $\gamma_{t_k}$ $\epsilon$-close to $\gamma_{-\infty}$ in $C^{0}$ norm. Then for any $t < t_k$, both $\gamma_t$ and $\gamma_{t_k}$ lie in the same hemisphere with $\interior{\Omega_{t_k}} \subsetneq \interior{\Omega_{t}}$. Thus $\gamma_t$ lies between $\gamma_{-\infty}$ and $\gamma_{t_k}$ and so is also $\epsilon$-close to $\gamma_{-\infty}$ in $C^0$ norm. 

$C^1$ convergence follows since the total length $L(t_k) \to 2\pi$ as $k\to\infty$. But also $\inpd[L]{t} = - \int k^2 ds < 0$ so that $L$ is monotone increasing backwards in time hence $L(t) \to 2\pi$ as $t\to-\infty$. But now parametrising $\gamma_t$ on $[0,1]$ with constant speed gives $\abs{\gamma_t'} = L(t) \to 2\pi$ and so $C^1$ convergence up to diffeomorphism follows.

Smooth convergence up to diffeomorphism now follows since the curvature and all derivatives of curvature converge to $0$.
\end{proof}
\section{Ancient solutions are shrinking round circles}
\label{sec-5}
In this section, we prove that ancient, convex solutions to the curve shortening flow are shrinking round circles. 

\begin{lemma}
[Backwards approximate symmetry]
\label{lem:backward_approximate_symmetry}
For any $\delta \in (0,\pi/2)$, there exists a $t_{\delta} \in (-\infty, 0)$ such that for every $V$ with $\delta(V) = \delta$ and all $t\leq t_{\delta}$, we have $\reflectionmap(\reflectionset{(\gamma_t)}^+) \geq \reflectionset{(\gamma_t)}^-$.
\end{lemma}

\begin{proof}

We use polar coordinates as in section \ref{sec:notation}. From Corollary \ref{cor:hemisphere} and Proposition \ref{prop:convex_sets}, we can assume that on $(-\infty, 0)$, $\gamma_t$ lies in the upper hemisphere $\{z>0\}$, written as a graph $\phi = f_t(\theta)$ of a smooth family of positive, smooth functions $f_t [0,2\pi] \to \RR$. Since $\gamma_t$ smoothly converges uniformly to the equator $\{z=0\}$, we have $\npd{\theta}{k} f_t \to 0$ uniformly for each $k\geq 0$. 

Provided that $\delta<\pi/4$, the reflected equator $\reflectionmap(\{z=0\})$ can be written as a graph $(\theta, g_{-\infty}(\theta))$. Since $\reflectionmap$ is an isometry, $\reflectionmap(\gamma_t)$ converges smoothly and uniformly to the reflected equator $\reflectionmap(\{z=0\})$. As the latter is a graph, possibly by choosing $t_0 < 0$ \emph{independently of $\delta$}, we can assume that $\reflectionmap(\gamma_t)$ may be written as a graph $(\theta, g_t(\theta))$ for $t < t_0$ with $g_t \to g_{-\infty}$ smoothly as $t\to-\infty$.

In spherical polar coordinates, for $X,Y \in \sphere^2$ the nearest-point projection is $(\theta(X), \phi(X)) \mapsto (\theta(X), 0)$. If $\theta(X) = \theta(Y)$, the statement $X\geq Y$ is equivalent to $\phi(X) \geq \phi(Y)$. Thus to show that $\reflectionmap(\gamma_t)^+ \geq \gamma_t^-$ it is enough to show that $g_t(\theta) \geq f_t(\theta)$. 

The proof is composed of estimates for \emph{interior points} (i.e. points away from $P \cap \sphere^2$) and for \emph{boundary points} (i.e. points near $P \cap \gamma_t$).

\emph{Interior Points}

For $\delta < \pi/4$, the reflected equator $(\theta, g_{-\infty}(\theta))$, $\theta \in [0,\pi]$ is given by a non-negative, smooth, concave function $g_{-\infty}$ symmetric about $\pi/2$ and strictly positive for $\theta \in (0, \pi)$. Given any $\epsilon \in (0, \pi/2)$, let $G = g_{-\infty} (\epsilon) = g_{-\infty}(\pi-\epsilon)$. Then $G < g_{\infty}(\theta)$ for any $\theta \in (\epsilon, \pi-\epsilon)$. Choose $t_1<t_0$ such that for $\theta \in (\epsilon, \pi-\epsilon)$ and $t<t_1$, we have $f_t(\theta) < G/2$ and $\abs{g_t(\theta) - g_{-\infty}(\theta)} < G/2$. This is possible since $f_t \to 0$ uniformly, and $g_t \to g_{-\infty}$ uniformly. Then, since $g_{-\infty} > G$ on $(\epsilon, \pi - \epsilon)$, for any $\epsilon > 0$, there is a $t_1 = t_1(\epsilon)$ such that $g_t(\theta) > f_t(\theta)$ for $\theta \in (\epsilon, \pi-\epsilon)$ and $t\leq t_1$.

\emph{Boundary Points}

Choose an orientation on $\theta$ so that $P \intersect \{z>0\}$ lies in the region with $\theta \in (-\pi, 0)$. Then recalling that $\gamma_t$ is a graph over the equator, $\gamma_t \intersect P = \{\theta_0(t), \theta_1(t)\}$ with $\theta_0(t) \in (-\pi, 0)$ and $\theta_1(t) \in (\pi, 2\pi)$. Moreover as $t \to -\infty$ we have $\theta_0(t) \to 0$ and $\theta_1(t) \to \pi$. The aim is to show that given $\tilde{\epsilon}>0$, there is a $t_{\tilde{\epsilon}}$ such that $g_t > f_t$ on $(\theta_0(t), \tilde{\epsilon}) \union (\pi-\tilde{\epsilon}, \theta_1(t))$ for every $t<t_{\tilde{\epsilon}}$. It's enough to prove it on $(\theta_0(t), \tilde{\epsilon})$. The proof on $(\pi-\tilde{\epsilon}, \theta_1(t))$ is similar.

We use that $f_t \to 0$, and $g_t \to g_{-\infty}$ smoothly and uniformly, and that $\theta_0(t) \to 0$ as $t\to-\infty$. Notice that $\reflectionmap(\{z=0\})$ lies above the equator $\{z=0\}$ for $\theta \in (0,\pi)$ and lies below for $\theta \in (-\pi,0)$. Thus, $g_{-\infty}$ is odd about $\theta=0$ and increasing near $\theta = 0$ so that $g_{-\infty}'(0) > 0$ (in fact equal to $\tan(2\delta)$) and $g_{-\infty}''(0) = 0$.

Choose $t_1 < t_0$ such that $g_t'(\theta_0(t)) > f_t'(\theta_0(t))$ for all $t<t_1$ which we can do since $f_t' \to 0$ uniformly and $g_t'(\theta_0) \to g_{-\infty}'(0) = \tan(2\delta) > 0$. We also have that $g_t(\theta_0) = f_t(\theta_0)$ since $\theta_0$ is the point about which $f_t$ is reflected across $P$. Expand $g_t$ and $f_t$ in a Taylor series about $\theta_0(t)$ to get that for $\theta > \theta_0$, $g_t > f_t$ if and only if
\[
g_t'(\theta_0) - f_t'(\theta_0) > - \frac{1}{2} (g_t''(c) - f_t''(c)) (\theta - \theta_0)
\]
where $c = c(\theta, t) \in (\theta_0(t), \theta)$. As $t\to -\infty$ the left hand side converges to $\tan(2\delta)$ whilst the right hand side converges to $0$ hence there is a $t_2  = t_2(\tilde{\epsilon}) < t_1$ such that the inequality $g_t - f_t > 0$ is satisfied for any $\theta \in (\theta_0(t), \tilde{\epsilon})$ and $t<t_2$.

\emph{Combined estimates}

To finish the proof, fix any $\epsilon>0$ and use the interior estimates to obtain $g_t > f_t$ for $\theta \in (-\epsilon, \pi-\epsilon)$ and $t < t_1$. Then choose $\tilde{\epsilon} > \epsilon$ to obtain $g_t > f_t$ for $\theta \in (\theta_0(t), \tilde{\epsilon}) \union (\pi-\tilde{\epsilon}, \theta_1(t))$ and $t < t_2$ from the boundary estimates. Then let $t_{\delta} = \min\{t_1, t_2\}$ to get $g_t > f_t$ for all $\theta \in (\theta_0(t), \theta_1(t))$ and all $t<t_{\delta}$.
\end{proof}

\begin{lemma}
[Approximate symmetry preserved]
\label{lem:approximate_symmetry_preserved}
There is a $T \in (-\infty, 0)$ such that $\reflectionmap(\gamma_t)^+ \geq \gamma_t^-$ for $t \in (-\infty, T)$ and all $\delta \in (0,\pi/4)$.
\end{lemma}

\begin{proof}
Recall that both the $\gamma_t^-$ and $\reflectionmap(\gamma_t)^+$ may be written as graphs over the equator for $t\in(-\infty, t_0)$. Since both the equator $\gamma_{-\infty}$ and the reflected equator $\reflectionmap(\gamma_{-\infty})$ meet $P$ transversely, $\gamma_t$ smoothly converges to the equator, and $\reflectionmap$ is an isometry, there is a $t_1 = t_1(\delta) \in (-\infty,t_0)$ such that both $\gamma_t^-$ and $\reflectionmap(\gamma_t)^+$ meet $P$ transversely for all $t \in (-\infty, t_1)$. Thus $\gamma_t^-$ and $\reflectionmap(\gamma_t)^+$ are connected curves meeting $P$ transversely in precisely two points for each $t$. 

Now we apply the maximum principle. Since $\reflectionmap$ is an isometry, $\reflectionmap(\gamma_t^+)$ evolves by curve shortening. Since $P\intersect \sphere^2$ is a great circle, it is stationary under the curve shortening flow so we can think of it too evolving by curve shortening. Therefore, as both $\gamma_t^-$ and $\reflectionmap(\gamma_t)^+$ meet $P$ transversely, the maximum principle ensures that both curves do not intersect $P$ at any \emph{other} points, hence remain in $\reflectionset{\sphere^2}^-$ for all $t\in(-\infty, t_1)$. 

The above allows us to set up a maximum principle argument: we have two connected curves $\gamma_t^-, \reflectionmap(\gamma_t^+)$ evolving by curve shortening and they agree at their end points which remain on $P$. By Lemma \ref{lem:backward_approximate_symmetry}, for all $t \in (-\infty, t_{\delta})$ we have $d(x,y,t) > 0$ for any $x\in\gamma_t^-$ and $y\in\reflectionmap(\gamma_t^+)$ away from the end points. We also obtain that at the end points, the angle $\reflectionmap(\gamma_t^+)$ makes with the $\{z=0\}$ plane is strictly bigger that the angle $\gamma_t^-$ makes with the $\{z=0\}$ plane. By the parabolic Hopf boundary point lemma (see e.g. \cite{MR1483984}), this positive lower bound is preserved under the flow and so $d(x,y,t) > 0$ for $(x,y)$ near both end points. Now, in the usual way (e.g. \cite{MR1656553}) a contradiction is obtained if $d(x,y,t) \leq 0$ at some time $t$ for some $(x,y)$ since this must occur at a first time $t>t_{\delta}$ at an interior point $(x,y)$.

This furnishes us with a $T_{\delta}$ for each $\delta$ such that $\reflectionmap(\gamma_t)^+ \geq \gamma_t^-$ for $t \in (-\infty, T_{\delta})$. Let $T = \inf\{T_{\delta}: \delta \in (0,\pi/2)\}$. We need to show that $T>-\infty$. To see this, observe that the above argument is valid provided both $\gamma_t^-$ and $\reflectionmap(\gamma_t)^+$
\begin{enumerate}
\item are graphs over the equator,
\item meet $P$ transversely,
\item are non-empty
\end{enumerate}
for $t \in (-\infty, T)$.

\begin{enumerate}
\item Recall that $\reflectionmap(\gamma_{-\infty})$ is a graph $g_{-\infty}^{\delta}$ for each $\delta\in(0,\pi/4)$ with maximum derivative at the end points $\theta = \{0, \pi\}$. As $\delta \to 0$, the derivative $(g_{-\infty}^{\delta})'(0)$ monotonically decreases to $0$. For each fixed $t$ then $[\reflectionmap(\gamma_t)+]'$ becomes more horizontal as $\delta$ decreases hence if $\reflectionmap(\gamma_t)^+$ is a graph (so does not have a vertical tangent) for some $\delta_0$, then it remains a graph for every $\delta < \delta_0$. Of course whether $\gamma_t^-$ is a graph or not is independent of $\delta$, and by convexity we know that $\gamma_t$ is a graph for all $t \in (-\infty,0)$.

\item The angle $P$ makes with the $\{z=0\}$ plane increases monotonically as $\delta \to 0$. If $\gamma_t^-$ and $\reflectionmap(\gamma_t)^+$ meet $P$ transversely for some $\delta_0$ then they continue to do so for every $\delta<\delta_0$.

\item Provided $\gamma_t^-$ lies below the maximum $\phi$ coordinate of $P \intersect \sphere^2$, both curves $\gamma_t^-$ and $\reflectionmap(\gamma_t)^+$ are non-empty. But now just observe that this $\phi$ monotonically increases to $\pi/2$ as $\delta \to 0$.
\end{enumerate}
\end{proof}

Next we characterise those curves $\alpha$ with maximal approximate symmetry for every $\delta>0$ as round circles.

\begin{prop}
[Exact symmetry]
\label{prop:exact_symmetry}
Let $\alpha \subset (\sphere^2)^+ = \sphere^2 \intersect \{z \geq 0\}$ be a smooth curve. If $\reflectionmap(\reflectionset{\alpha}^+) \geq \reflectionset{\alpha}^-$ for every $\reflectionvector$ such that $\ip{\reflectionvector}{\vec{e}_z} < 0$ and $\delta(\reflectionvector) \in (0,\pi/4)$, then $\alpha$ is a round circle with center the north pole $(0,0,1)$.
\end{prop}

\begin{proof}
Let $\reflectionvector_0$ be a vector in $\RR^3$ such that $\ip{\reflectionvector_0}{\vec{e}_z} = 0$ (so that $\delta(\reflectionvector_0) = 0$). Choose any $X \in \reflectionset{(\gamma_{-\infty})}^-$. Then by assumption, we have 
\[
\reflectionmap(\reflectionset{\alpha}^+) \geq_X \reflectionset{\alpha}^-
\]
for all $\reflectionvector$ lying in the plane spanned by $\vec{e}_z$ and $\reflectionvector_0$, and with $\ip{\reflectionvector}{\vec{e}_z} < 0$ and $\delta(\reflectionvector) \in (0,\pi/4)$. By continuity, letting $\reflectionvector \to \reflectionvector_0$ we obtain $\reflectionmap[\reflectionvector_0](\reflectionset[\reflectionvector_0]{\alpha}^+) \geq_X \reflectionset[\reflectionvector_0]{\alpha}^-$ for each $X \in \gamma_{-\infty}^-$ and hence 
\[
\reflectionmap[\reflectionvector_0](\reflectionset[\reflectionvector_0]{\alpha}^+) \geq \reflectionset[\reflectionvector_0]{\alpha}^-
\]
for every $\reflectionvector_0$ with $\ip{\reflectionvector_0}{\vec{e}_z} = 0$. 

Now, we need some simple properties of $\reflectionmap[\reflectionvector_0]$ following from the fact that $\ip{\reflectionvector_0}{\vec{e}_z} = 0$:
\begin{itemize}
\item $\reflectionmap[\reflectionvector_0]^2 = \id$,
\item $\alpha \geq \beta \Rightarrow \reflectionmap[\reflectionvector_0](\alpha) \geq  \reflectionmap[\reflectionvector_0](\beta)$,
\item $\reflectionmap[\reflectionvector_0] = \reflectionmap[-\reflectionvector_0]$, and
\item $\alpha^{\pm}_{\reflectionvector_0} = \alpha^{\mp}_{-\reflectionvector_0}$.
\end{itemize}

Thus we obtain,
\begin{align*}
\reflectionset[\reflectionvector_0]{\alpha}^+ &= \reflectionmap[\reflectionvector_0]^2(\reflectionset[\reflectionvector_0]{\alpha}^+) = \reflectionmap[\reflectionvector_0](\reflectionmap[\reflectionvector_0](\reflectionset[\reflectionvector_0]{\alpha}^+)) \\
&\geq \reflectionmap[\reflectionvector_0](\reflectionset[\reflectionvector_0]{\alpha}^-) = \reflectionmap[-\reflectionvector_0](\reflectionset[-\reflectionvector_0]{\alpha}^+) \\
&\geq \reflectionset[-\reflectionvector_0]{\alpha}^- = \reflectionset[\reflectionvector_0]{\alpha}^+.
\end{align*}
We must have equality all the way through and hence 
\begin{equation}
\label{eq:reflection_invariant}
\reflectionmap[\reflectionvector_0](\reflectionset[\reflectionvector_0]{\alpha}^-) = \reflectionset[\reflectionvector_0]{\alpha}^+
\end{equation}
for any $\reflectionvector_0$. 

To finish, equation \eqref{eq:reflection_invariant} implies that $\alpha$ must have a horizontal (i.e. no $\vec{e}_z$ component) tangent at $\reflectionplane[\reflectionvector_0] \intersect \alpha$. But every point of $\alpha$ lies on $\reflectionplane[\reflectionvector_0]$ for some $\reflectionvector_0$ hence $\alpha$ has a horizontal tangent everywhere and hence is a round circle.
\end{proof}

\begin{theorem}
Let $\gamma_t$ be a convex, ancient solution to the curve shortening flow. Then $\gamma_t$ is the unique up to isometry of $\sphere^2$, family of shrinking circles ancient solution.
\end{theorem}

\begin{proof}
The approximate symmetry preserved lemma \ref{lem:approximate_symmetry_preserved}, implies that for every $\reflectionvector$ with $\delta(\reflectionvector) \in (0,\pi/4)$, $\reflectionmap(\reflectionset{(\gamma_t)}^+) \geq \reflectionset{(\gamma_t)}^-$ for all $t\in (-\infty,T)$. The exact symmetry proposition \ref{prop:exact_symmetry} applies at each such $t \in (-\infty, T)$, showing that $\gamma_t$ is a round circle for every $t\in(-\infty,T)$. Uniqueness of solutions ensures that $\gamma_t$ is a round circle for every $t \in (-\infty,0)$.
\end{proof}
\printbibliography
\end{document}